\documentclass{amsart}
\usepackage{color}
\newtheorem{theorem}{Theorem}[section]

\newtheorem{lemma}[theorem]{Lemma}
\newtheorem{proposition}[theorem]{Proposition}
\theoremstyle{definition}
\newtheorem{definition}[theorem]{Definition}

\theoremstyle{remark}
\newtheorem{remark}[theorem]{Remark}
\numberwithin{equation}{section}

\newcommand{\sign}[1]{\mathrm{sgn}(#1)}

\usepackage[all,cmtip]{xy}

\usepackage{bm}
\usepackage{amssymb}

\usepackage{marvosym}%corresponding author
\usepackage{ifsym}

\makeatletter
\renewcommand{\section}{\@startsection{section}{1}{0mm}
 {-\baselineskip}{0.5\baselineskip}{\bf\leftline}}
\makeatother

\begin{document}

\title[Adaptive Fourier decomposition of slice regular functions]{Adaptive Fourier decomposition of slice regular functions}

\author[M. Jin]{Ming Jin}
\author[I. T. Leong]{Ieng Tak Leong}
\author[T. Qian]{Tao Qian*}
\author[G. B. Ren]{Guangbin Ren}
\address{Ming Jin, school of Mathematical sciences, Fudan University,
 Shanghai 200433, China}
\email{jinming$\symbol{64}$fudan.edu.cn}

\address{Ieng Tak Leong, Department of Mathematics, Faculty of Science and Technology, University of Macau, Macao, China}
\email{itleong$\symbol{64}$umac.mo}

\address{Tao Qian (Corresponding author),
Macau Institute of Systems Engineering, Macau University of Science and Technology, Macau, China}
\email{tqian$\symbol{64}$must.edu.mo}

\address{Guangbin Ren, Department of Mathematics, University of Science and
Technology of China, Hefei 230026, China}
\email{rengb$\symbol{64}$ustc.edu.cn}

\thanks{This work was supported by the Science and Technology Development Fund, Macau SAR (File no. 0123/2018/A3), University of Macau MYRG 2018-00168-FST \\
*Corresponding author: Tao Qian}

\begin{abstract}
 In the slice Hardy space over the unit ball of quaternions, we introduce
the slice hyperbolic backward shift operator $\mathcal S_a$ with the decomposition process
$$f=e_a\langle f, e_a\rangle+B_{a}*\mathcal S_a f,$$
where $e_a$ denotes the slice normalized Szeg\"o kernel and $ B_a $ the slice Blaschke factor.
Iterating the above decomposition process, a corresponding maximal selection principle gives rise to the slice adaptive Fourier decomposition. This leads to a adaptive slice Takenaka-Malmquist orthonormal system.
%\textcolor{red}{If the maximum selection principle is not used in the algorithm, it yields the Beurling-Lax theorem about the decomposition of the slice Hardy space into shift invariant subspaces.}

\end{abstract}
\maketitle

\section{Introduction}

The purpose of this article is to introduce the quaternionic slice hyperbolic backward shift operators %$\mathcal S_a$
in the slice Hardy space $H^2(\mathbb B)$   of the unit ball of quaternions.
 At each of the process we decompose a function $f\in H^2(\mathbb B)$ into an orthogonal sum of two functions of which one is in the subspace generated by a slice normalized Szeg\"o kernel $e_a$ and the other is in its orthogonal complement expressed by the Blaschke factor and the backward shift operator $\mathcal S_a$.
By iterating the process we decompose a given function
$f\in H^2(\mathbb B)$ into a slice Takenaka-Malmquist orthonormal system.
%\textcolor{red}{It also yields the Berling-Lax theorem about the decomposition of $H^2(\mathbb B)$   into shift invariant subspaces.}

Our motivation comes from adaptive Fourier decompositions (AFDs) for the holomorphic Hardy spaces of the unit disc and of a half of the complex plane \cite{Qian_1}.
The type of decompositions provides
approximations by suitable linear combinations of
parameterized reproducing kernels in the respective Hardy spaces. The decompositions may result in merely an orthonormal system: It may not be a basis but adaptive to the given signal.
It, however, achieves fast decomposition through extracting out the greatest energy portion from the orthogonal remainder at each iterative step.
Together with the process there arise the Takenaka-Malmquist (TM) orthonormal systems (\cite{Walsh}). If instead of the maximal selection principle one uses a set of parameters satisfying the hyperbolic non-separability rule $\sum_{k=1}^\infty (1-|a_k|)=\infty,$ then the decomposition process results in a TM basis. The Fourier basis $\{z^n\}_{n=0}^\infty$ is a particular case corresponding to all the parameters $a_n$ being identically zero.
Since Takenaka-Malmquist system consists of
rational functions, the study falls into the scope of rational approximation (\cite{Walsh}).  The Takenaka-Malmquist bases can be thought of hyperbolic versions of the Fourier system.
The adaptive Fourier decomposition allows repeating selections of the parameters that offers attainability of the best matching pursuit at each step of parameter selection.
 The adaptive Fourier decomposition methodology facilitates efficient and thus useful sparse representations. It is, in particular, effective when in the underlying Hilbert space there does exist an approximation theory. % \cite{Alpay_2}.
Adaptive Fourier decomposition with different contexts has undergone substantial developments (\cite{Alpay_1,Alpay_2,Qian_3,Qian_5}) with ample applications, such as digital signal processing \cite{Qian_4}, image processing \cite{LZQ_1}, and system identification \cite{WQLG_1}.

In this paper we establish the slice adaptive Fourier decomposition of slice regular functions over quaternions.
It is a higher dimensional extension of the subject in the one complex variable case.
 The slice regular function theory of quaternions was initiated  by Gentili and Struppa \cite{Gentili_2}, and soon developed by a number of researchers. See, for instance \cite{Colombo_1,Alpay_3,AFS_book,Ghiloni_1,Gentili_B}. This theory generalizes
the holomorphic theory of one complex variable to quaternions.
 It is remarkable that the slice regular function theory relies on the slice structure of quaternions, which is different from the monogenic function theory over Clifford algebras. The great difference among the slice analysis, Clifford analysis, several complex variables is due to the canonical topology in slice analysis distinct to the Euclidean topology;
see \cite{Dou_1}.
The slice theory shows vigorous vitality in non-commutative Clifford algebra \cite{Colombo_1}, and non-associative real alternative algebras \cite{Ghiloni_1}, \cite{Ren_3}. It also has significant applications  in
differential geometry \cite{Gentili_4}, geometric function theory \cite{Ren_1},  and operator theory \cite{CSS_book}.

To state our main results we provide some preliminaries of the slice Hardy space which are mostly adopted from \cite{Alpay_3, AFS_book}. Here we introduce an equivalent definition.
The slice Hardy space $ H^2( \mathbb B) $ over the unit ball of the quaternions  consists  of the slice regular  functions $f:\mathbb B\to\mathbb H$ satisfying
 \[ \|f\|:=\Big(\frac{1}{4\pi^2}\int_{\partial \mathbb B}\frac{1}{|Im(q)|^2}|f(q)|^2d\sigma\Big)^{1/2}<\infty,   \]
where $ d\sigma $ is the Lebesgue surface measure  on $\partial \mathbb B$.
The polarization identity of this norm provides $ H^2( \mathbb B) $ with an inner product so that
it becomes
 a quaternionic Hilbert space with
 reproducing kernel.
   Its  normalized reproducing kernel is called  the slice normalized Szeg\"o kernel,  defined as
\begin{equation*}\label{e_aa}
e_a(q):=\sqrt{1-|a|^2}(1-q\bar a)^{-*},
\end{equation*}
 for   a parameter $ a \in \mathbb B $ and any $q\in\mathbb B,$
 where the $*$-product is defined by
 $$(f*g)(q)=\sum_{n=0}^\infty q^n \sum_{k=0}^na_k b_{n-k},$$
 for any  two slice regular functions
 $$f=\sum_{n=0}^\infty q^n a_n   \quad \mbox{and} \quad g=\sum_{n=0}^\infty q^n b_n,$$
 where all  $a_n, b_n\in\mathbb H.$ Thus, $f^{-*}$ denote as the inverse of $f$ under this $*$-product.

 Based on the $*$-product, we introduce
 the quaternionic slice hyperbolic backward shift operators $\mathcal S_a$, which is uniquely determined by the identity
 \begin{eqnarray}\label{def:shift-101}
 f=e_a\langle f, e_a\rangle+B_{a}*\mathcal S_a f,
 \end{eqnarray}
 for any $a\in\mathbb B$.  Here   $B_a$ is the slice M\"obius transformation, or an order-1 Blaschke product:
\[B_a(q):= (1-q\bar a)^{-*}*(a-q)\frac{a}{|a|}. \]
It is noted that the notion of the Hardy spaces and the Blaschke factors in relation to slice regular functions have been introduced and studied by researchers, first appearing in \cite{Alpay_3}, and also others, which are summarized in the book \cite{AFS_book}.

Iterating the construction in (\ref{def:shift-101}),
we achieve an algebraic relation
%a greedy algorithm, i.e. the following  approximations
\begin{eqnarray}\label{def:TM-102}f=\sum_{j=1}^nT_{j}\langle f, T_{j}\rangle +B_n*
\big(\mathcal S_{a_n}\circ  \cdots \circ\mathcal S_{a_1}f\Big)\end{eqnarray}
for arbitrary $a_1, \ldots, a_n\in \mathbb H$, where we denote
$$T_n=B_{n-1}*e_{a_n},$$
and
$$B_n=B_{a_1}*B_{a_2}*\cdots*B_{a_{n}}.$$
Here $\{T_{n}\}_{j=1}^n$ is defined to be the  slice
Takenaka-Malmquist  system which constitutes an slice orthonormal system in $H^2(\mathbb B)$ associated with the non-orthogonal set $\{e_{a_j}\}_{j=1}^n$.
We point out that in the classical case a TM system is identical with the Gram-Schmidt (GS) orthogonalization applied to the corresponding Szeg\"o kernels. However,  in our slice case,  the GS process is no longer valid since the product among slice regular functions is now $*$-product.
This $*$-product brings challenge of verifying the orthogonality of a TM system.

Decomposition (\ref{def:TM-102})
 can be restated as a Beurling-Lax type relation: %the Beurling-Lax theorem:
 \begin{eqnarray}\label{def:Beurling-Lax-103}
 H^2(\mathbb B)=\mbox{span}\{T_1, \ldots, T_n\} \oplus B_n *H^2(\mathbb B).\end{eqnarray}
By applying the maximum selection principle, we obtain the slice  $AFD$ for quaternionic slice Hardy space functions.

The paper is organized as follows.
In \S 2, we recall some basic concepts and results of slice regular functions and the foundation of  the slice Hardy space.
%In \S 3, we introduce the slice Hardy space and its dictionary.
In \S 3, the slice Takenaka-Malmquist orthonormal system is established.
In \S 4, We introduce the iterative process, the adaptive selecting process, and prove convergence of  the slice adaptive Fourier series.
In \S 5 a convergence rate result is proved.
The case of slice Hardy space of the right half plane is outlined in \S 6.

\section{Preliminary}
\subsection{Slice regular functions}

This paper works on  slice regular functions  over the non-commutative quaternionic field \cite{Gentili_B}. Firstly, the quaternionic field $ \mathbb H $ is linearly generated by an orthogonal basis $ \{1, e_1,e_2,e_3:=e_1e_2\} $ of $ \mathbb R^4 $ with the following multiplication rule:
\[ e_ie_j+e_je_i=-2\delta_{ij}, \quad\quad i,j=1, 2, 3,\]
where $ \delta_{ij} $ equals $ 1 $ if $ i=j $ and $ 0 $ otherwise.

Now we recall some   definitions and results of the slice regular function theory.
This theory is based on the slice structure of  $ \mathbb H $, i.e.
$$ \mathbb H=\bigcup_{I \in \mathbb S} \mathbb C_I, $$
where  $ \mathbb S $ denotes  the set of imaginary units of  $ \mathbb H $, namely,
\[ \mathbb S:=\{ q \in \mathbb H  \  \bm{\big\vert} \ q^2=-1 \},  \]
and $ \mathbb C_I $ denotes the slice of $ \mathbb H $ made through $ I $, i.e.,
$$ \mathbb C_I:=\{x+yI, \quad x,y\in \mathbb R\}. $$
According to the slice structure, any $ q \in \mathbb H $ can be written as $ q=x+yI $ with $ x,\ y \in \mathbb R $ and $ I \in \mathbb S $.

\begin{definition}[slice function]\label{slice}
%Let $ f $ be a quaternion-valued function defined on a set  $ \Omega\subset \mathbb H $ for which it satisfies
Let $ \Omega $ be a set in $ \mathbb H $, and $ f $ a quaternion-valued function defined on $ \Omega $ that satisfies
\[f(x+yJ)=\frac{1}{2}\big(f(x+yI)+f(x-yI)\big)+\frac{JI}{2}\big(f(x-yI)-f(x+yI)\big),\]
provided  $ x,\ y \in \mathbb R $ and  $ I, \ J \in \mathbb S $ such that  $ x\pm yI $ and $ x+yJ $ belong to $ \Omega .$
Then the function $ f $ is said to be a slice function on $\Omega$.
\end{definition}

\begin{definition}[slice regular function]
Let $ f $ be a  slice function  defined on a  domain $ \Omega \subset \mathbb H $. For each $ I \in \mathbb S $, let
$ \Omega_I:=\Omega\cap \mathbb C_I $ and $ f_I:=f|_{\Omega_I} $ be the restriction of $ f $ to $ \Omega_I $. The restriction $ f_I $ is said to be holomorphic if it has continuous partial derivatives and
\[\bar \partial_I f_I(x+yI)=\frac{1}{2}\big(\partial_x+I\partial_y\big)f_I(x+yI\big)=0.\]
If for each $ I\in \mathbb S $, $ f_I $ is holomorphic in $ \Omega_I $, then $ f $ is called a slice (left) regular function.
\end{definition}

\begin{remark}
 Similarly, we can define slice right regular functions.
 \end{remark}

\begin{lemma}[splitting]
Let $ I \in \mathbb S $ and $ \Omega_I $ be open in $ \mathbb C_I $. The function $ f_I: \Omega_I \to \mathbb H  $ is holomorphic if and only if, for all $ J\in \mathbb S $  with $ J \perp I $ and every $ z=x+yI $, there holds
\[ f_I(z)=F(z)+G(z)J,\]
where $ F, \ G:\Omega_I \to \mathbb C_I $ are complex-valued  holomorphic functions of one complex variable.
\end{lemma}

In the slice regular function theory, under the usual multiplication the product of two slice regular functions is no longer slice regular in general. So there comes the $*$-product.

\begin{definition}[$*$-product]\label{*}
Let $ \mathbb B$ be the Euclidean unit ball
of $ \mathbb H $. Let $ f, \ g: \mathbb B \to \mathbb H $ be  slice regular functions and let $ f(q)=\sum_{n\in\mathbb N}q^n a_n, \ g(q)=\sum_{n\in\mathbb N}q^n b_n$ where $ a_n, b_n \in \mathbb H $ be their power series expansions. The $*$-product of $ f $ and $ g $ is the slice regular function defined by  %(\cite{AFS_book}, also in \cite{Gentili_B}, Definition 1.26)
( \cite{Gentili_B}, Definition 1.26)
\[ (f*g)(q)=\sum_{n\in\mathbb N}q^n \sum_{k=0}^na_kb_{n-k}. \]
The regular conjugate of $ f $ is the slice regular function defined by
\[ f^c(q)=\sum_{n\in\mathbb N}q^n \bar a_n.\]
The symmetrization of $ f $ is defined to be  the function
\[f^s=f*f^c=f^c*f.\]
Furthermore, if $ f \ne 0 $, the regular reciprocal of $ f $ is the function defined on
$ \mathbb B\setminus Z_{f^s} $ as
\[ f^{-*}=\frac{1}{f^s}f^c,\]
where $ Z_{f^s} $ is the zero set of $ f^s $.
\end{definition}

The $*$-product of two slice regular functions is slice regular and it is related to the usual multiplication through the following relations (see \cite{Gentili_B}, Theorem 3.4):
\begin{equation}\label{z_wan}
(f*g)(q)=\begin{cases} f(q)g(\tilde q), \quad &\textit{if }  f(q)\ne 0, \\
0, \quad & \mbox{otherwise}, \end{cases}
\end{equation}
where $\tilde q= f^{-1}(q)q f(q)\in [q] $ with $ [q] $  the symmetry of $ q=x+yI $ defined by
 \[ [q]:=\{x+yJ  \  \bm{\big\vert} \ J\in \mathbb S \}. \]
 Furthermore,
 \begin{equation}\label{z_wanw}
f^{-*}*g(q)=f^{-1}(\hat q)g(\hat q), \quad  \forall q\in \mathbb B\setminus Z_{f^s},
\end{equation}
where $ \hat q= f^c(q)^{-1}q f^c(q)\in [q] $.

A set $ \Omega\subset \mathbb H $ is said to be axially symmetry if  for any point  $q\in\Omega $,  there holds $ [q]\subset\Omega $.

%\begin{remark}
%The domain of definition in Definition \ref{*} can extend to any axially symmetric domain.
%\end{remark}

\begin{definition}Let $\Omega $ be an axially symmetric domain in $ \mathbb H $. A slice regular function $ f:\Omega \to \mathbb H $ such that $ f(\Omega_I)\subset \mathbb C_I $ for all $ I\in \mathbb S $ is called a slice preserving function.
\end{definition}

\begin{remark}
If $ f$ is a slice function on an axially symmetric domain $ \Omega$, its symmetrization function $ f^s:\Omega \to \mathbb H $ is  a slice preserving function.
\end{remark}

\begin{theorem}[Cauchy's formula]\label{Cauchy}
Let $ f $ be a slice regular function on an open set $ \Omega\subset \mathbb H $. If $ U $ is a bounded axially symmetric open set with $ \bar U \subset \Omega $ where $ \bar U $ is the closure of $ U $ and if $ \partial U_I $ for $ I\in \mathbb S $ is a finite union of disjoint rectifiable Jordan curves, then for $ q\in U $,
\[ f(q)=\frac{1}{2\pi}\int_{\partial U_I }(s-q)^{-*}ds_If(s).\]
where $ ds_I=-Ids $.
\end{theorem}

\subsection{The foundation of the slice Hardy space over $ \mathbb B $}
In this subsection, we recall the precondition of the slice Hardy space over the unit ball  \cite{AFS_book}.
Let  $ \mathbb B $ be the Euclidean unit ball of $ \mathbb H $ and  $ \mathbb T:=\partial \mathbb B $ its boundary.
For any $  I\in \mathbb S $, denote
  $\mathbb T_I:=\mathbb T \cap \mathbb C_I $.
Let $ L^2(\mathbb T_I) $    be the function space  consisting of  Lebesgue measurable slice  functions $ f $ defined  on $ \mathbb T $  for which
\[ \frac{1}{2\pi}\int_0^{2\pi}|f_I(e^{It})|^2dt<\infty.  \]
The splitting lemma provides a power series expansion of $ f_I $, i.e.
\begin{equation}\label{FFII}
f_I(e^{It})=\sum\limits_{k=-\infty}^{\infty}e^{Ikt}a_k,
\end{equation}
where $ a_k \in \mathbb H $  satisfies
\[\sum\limits_{k=-\infty}^{\infty}|a_k|^2< \infty.\]
For any $ f,\ g \in L^2(\mathbb T_I) $,  the  inner product $$ \left< \cdot, \cdot \right>: L^2(\mathbb T_I)   \times L^2(\mathbb T_I)  \rightarrow \mathbb H $$ is defined by
\begin{equation}\label{inner}
 \left< f_I, g_I \right>:=\frac{1}{2\pi}\int_{0}^{2\pi}\overline{ g_I(e^{It})}f_I(e^{It})dt.
\end{equation}

It is easy to verify that $ \left< \cdot, \cdot \right> $ is  an inner product. i.e.
for any $ f,g,h\in L^2(\mathbb T_I)$ and  $ \lambda, \mu \in \mathbb H $, there hold
\begin{itemize}
\item $  \left< f_I\lambda+g_I\mu,h_I \right>= \left< f_I,h_I \right>\lambda+ \left< g_I,h_I \right>\mu.$
\item $  \left< f_I, g_I \right>=\overline{\left< g_I, f_I \right>}.$
\item $ \left< f_I, f_I \right>\geqslant 0$, where the equality holds if and only if $ f_I=0 $.
\end{itemize}
Furthermore, there holds the Cauchy-Schwarz inequality, i.e.
\begin{equation}\label{C-S}
 |\left< f_I, g_I \right>|^2\leqslant \left< f_I, f_I \right>\left< g_I, g_I \right>.
 \end{equation}

The power series expansion in (\ref{FFII}) shows that $ L^2(\mathbb T_I) $ is a right $ \mathbb H $-module. Thus, $ L^2(\mathbb T_I) $ equipped with the inner product $ \left< \cdot, \cdot \right> $ is a right $\mathbb H$-module inner product space.

Denote
\[ H_+^2( \mathbb T_I):=\{  f_I(e^{It})=\sum\limits_{k=0}^{\infty}e^{Ikt}a_k \ : \ \ a_k\in\mathbb H, \ \ \sum\limits_{k=0}^{\infty}|a_k|^2< \infty \}.\]
It is a closed subspace  of $ L^2(\partial \mathbb B_I) $.
Recall that the Hilbert transformation $$ \tilde H: L^2(\mathbb T_I)\to L^2(\mathbb T_I) $$ is defined by
\[\tilde H f_I(e^{It})=\sum\limits_{k=-\infty}^{\infty}(-I)\sign ke^{Ikt}a_k,\]
where $ a_0 $ is the  coefficient in formula (\ref{FFII}).
Thus, each $ f \in  H^2(\mathbb T_I) $ can be represented as
\[  f_I = \frac{a_0+f_I+I\tilde H f_I}{2}.\]
In fact, $ L^2(\mathbb T_I) $ has the following direct sum decomposition
\begin{equation}\label{HH--}
L^2(\mathbb T_I)=H_+^2( \mathbb T_I)\oplus H_-^2( \mathbb T_I),
\end{equation}
where
\[ H_-^2( \mathbb T_I):=\{  f_I(e^{It})=\sum\limits_{k=-\infty}^{-1}e^{Ikt}a_k:   \ \  a_k\in\mathbb H, \ \ \sum\limits_{k=0}^{\infty}|a_k|^2< \infty \}.\]

Thus, the function space $ H^2(\mathbb B_I) $ with  $  I\in \mathbb S $ is the function space  consisting of  slice regular functions $ f $ defined  in $ \mathbb B $ for which
\[ \|  f_I\|^2:=\sup\limits_{0\leqslant r<1}\frac{1}{2\pi}\int_0^{2\pi}|f_I(re^{It})|^2dt<\infty.   \]

For any $ f \in H^2(\mathbb B_I) $, define its radial limit
\begin{equation}\label{def:rad}
 \hat{f_I}(e^{It}):=\lim_{r\to1} f_I(re^{It}).
 \end{equation}
The limit $ \hat f_I $ exists almost everywhere.

\begin{theorem}\label{H2}
Let $ f\in H^2(\mathbb B_I) $ for some $ I \in \mathbb S $. The radial limit of $ f_I $ exists almost everywhere on  $ \mathbb T_I $. Furthermore, there is an isometric isomorphism
\begin{eqnarray*}
 H^2(\mathbb B_I)&\to& H_+^2( \mathbb T_I)
 \\
 f&\mapsto& \hat{f_I}.
 \end{eqnarray*}

\end{theorem}

%\begin{proof}
%Splitting lemma shows there exists  $ J \in \mathbb S $ with $ J \perp I $ such that
%\[f(re^{It})=F(re^{It})+G(re^{It})J, \]
%where $  F,\ G: \mathbb B_I \to \mathbb C_I$ are holomorphic functions.
%The Hardy space theory on the unit disc of  $ \mathbb C $ shows the existence of the radial limitations of holomorphic function $ F$, i.e.
%\[ \hat{F}(e^{It}):=\lim_{r\to1} F(re^{It}), \]
%as well as the equality
%\[\frac{1}{2\pi}\int_{0}^{2\pi}|\hat F(e^{It})|^2dt=\sup\limits_{0\leqslant r<1}\frac{1}{2\pi}\int_0^{2\pi}|F(re^{It})|^2dt.\]
%Thus, the limit $ \hat{f_I} $ exists, i.e.
%\begin{equation*}
%\begin{aligned}
%\hat{f_I}(e^{It})=&\lim_{r\to1} \big(F(re^{It})+ G(re^{It})J\big)
%\\
%=&\hat F(e^{It})+ \hat G(e^{It})J.
%\end{aligned}
%\end{equation*}
%Moreover, we obtain
%\begin{equation*}
%\begin{aligned}
%&\sup\limits_{0\leqslant r<1}\frac{1}{2\pi}\int_0^{2\pi}|f_I(re^{It})|^2dt
%\\
%&=\sup\limits_{0\leqslant r<1}\frac{1}{2\pi}\int_0^{2\pi}|F(re^{It})|^2dt+\sup\limits_{0\leqslant r<1}\frac{1}{2\pi}\int_0^{2\pi}|G(re^{It})|^2dt.
%\\
%&=\frac{1}{2\pi}\int_0^{2\pi}|\hat F(e^{It})|^2dt+\frac{1}{2\pi}\int_0^{2\pi}|\hat G(e^{It})|^2dt.
%\\
%&=\frac{1}{2\pi}\int_0^{2\pi}|\hat f_I(e^{It})|^2dt.
%\end{aligned}
%\end{equation*}
%Namely,
%\[ \|  f_I\|^2=\frac{1}{2\pi}\int_0^{2\pi}|\hat f_I(e^{It})|^2dt. \]
%\end{proof}

\begin{remark}
For any slice regular function $ f $ and for any $ I, J \in \mathbb S $,
$  f_I \in H^2( \mathbb B_I) $ if and only if  $   f_J \in H^2( \mathbb B_J) $.
\end{remark}

\begin{remark}
The $ L^2(\mathbb T_I)$ space can be expressed as direct sum of the two corresponding Hardy spaces, the latter consisting of boundary limits of well behaved holomorphic functions. Due to this relation, studies of functions of finite energy may use complex analysis methods. This shows the role and importance of Hardy space theory.
%In the theory of signal analysis, functions in $ L^2(\mathbb T_I) $ are complex   signals. Since functions in its closed subspace $ H^2( \mathbb B_I)$ possess non-negative  phase derivative, which is physically realizable signals. Thus, we can apply the holomorphic function theory as tools to deal with the signals which confirms the important role of the Hardy space in the signal processing.
\end{remark}

\section{Slice rational orthogonal system}

The slice Hardy space introduced in this section is equivalent to the definition in \cite{AFS_book}. The reproducing kernel and Blaschke products of the slice Hardy context have also been studied in the book. Thus, our main result in this section is the slice rational orthogonal system $ \{T_k\}_{k\geqslant 1} $, i.e. Theorem \ref{TM_Th}.

\begin{definition}
The slice Hardy space $ H^2(\mathbb B) $ consists of slice regular functions $ f $ defined  in $ \mathbb B $ which satisfies
\begin{eqnarray} \label{def:hardy-101}\|f\|^2:=\frac{1}{4\pi^2}\int_{\partial \mathbb B}\frac{1}{|Im(q)|^2}|f(q)|^2d\sigma(q)<\infty,
\end{eqnarray}
where $ d\sigma $ is the surface area element on $\partial \mathbb B$.
\end{definition}
Based on the slice technique and cylindrical coordinate transformation  \cite{Ghiloni_6}, we can polarize (\ref{def:hardy-101})   to one slice as following:
\begin{equation*}
\begin{aligned}
 \left < f, g \right>:=&\frac{1}{4\pi^2}\int_{T^2}\sin \theta_1d\theta \int_{\partial \mathbb B_{I(\theta)}}\overline{g(x+I(\theta)y)}f(x+I(\theta)y)dxdy
\\
=&\frac{1}{2\pi}\int_{T^2}\sin \theta_1d\theta\frac{1}{2\pi}\int_{0}^{2\pi}\overline{g(e^{I(\theta)t})}f(e^{I(\theta)t})dt
\\
=&\frac{1}{2\pi}\int_{T^2}\sin \theta_1d\theta \left < f_{I(\theta)}, g_{I(\theta)} \right>,
\end{aligned}
\end{equation*}
where $ \theta=(\theta_1, \theta_2)\in T^2:= [0,\pi]^2$ and
$$ I(\theta):=(e_1,e_2,e_3)\varphi(\theta)\in \mathbb S   $$
with
\begin{equation*}
\begin{aligned}
\varphi(\theta)=
&{\left( \begin{array}{l}
\cos \theta_1
\\
\sin \theta_1 \cos \theta_2
\\
\sin \theta_1  \sin \theta_2
\end{array} \right )}.
\end{aligned}
\end{equation*}

Notice that $ H^2(\mathbb B) $ is a reproducing kernel Hilbert space.
For any $ a \in \mathbb B $, define the slice normalized Szeg\"o kernel as
\begin{equation}\label{e_aa}
e_a(q):=e(a,q):=\sqrt{1-|a|^2}(1-q\bar a)^{-*},  \qquad \forall q\in \mathbb B,
\end{equation}
where $ -* $ is the regular reciprocal in Definition \ref{*}.
Since $  (1-q\bar a)^s $ does not have zero points, $e_a$  is  a left slice regular function over $ \mathbb B $.
Besides, the property of $*$-product shows the general conjugation of $ e_a $:
\[  \overline {e_a(q)}=\sqrt{1-|a|^2}(1-a\bar q)^{-*},  \qquad \forall q\in \mathbb B, \]
which is right conjugate slice regular over $ \mathbb B $.
We claim that $ e_a $ is the normalized reproducing kernel of $ H^2(\mathbb B) $. In fact,

\begin{equation}
\begin{aligned}
\left < f_I, (e_a)_I \right>=&\frac{1}{2\pi}\int_{0}^{2\pi}\overline{e_a(e^{It})}f(e^{It})dt
\\
=&\frac{\sqrt{1-|a|^2}}{2\pi}\int_{\partial \mathbb B_I}(1-ae^{-It})^{-*}\big(e^{-It}(-I)de^{It}\big)f(e^{It})
\\
=&\frac{\sqrt{1-|a|^2}}{2\pi}\int_{\partial \mathbb B_I}(1-a\bar q)^{-*}*\bar q(-Ids)f(q)
\\
=&\frac{\sqrt{1-|a|^2}}{2\pi}\int_{\partial \mathbb B_I}(q-a)^{-*}(-Ids)f(q)
\\
=&\sqrt{1-|a|^2}f(a),
\end{aligned}
\end{equation}
where the third equality holds because the function $  g(\bar q)=\bar q $ is a
(right) conjugate slice preserving function so that
 the $*$-product reduces to  the usual  product.
The last equality holds owing to the Cauchy integral formula (i.e. Theorem \ref{Cauchy}).
Furthermore, since $ \sqrt{1-|a|^2}f(a) $ is independent of the imaginary unit of $ q $, we obtain
\[ \left < f, e_a \right>=\sqrt{1-|a|^2}f(a). \]

\begin{remark}
We claim that the operator $ \mathcal S: L^2(  \mathbb T) \rightarrow  L^2(  \mathbb T) $ defined by
\[ \mathcal Sf(q):=\left <f, e(\cdot, q)\right > \]
is an orthogonal projection operator from $ L^2( \mathbb T) $  to  $ H^2( \mathbb B) $, so $ e_a$ with $ a \in \mathbb B $ is actually the Szeg\"o kernel of $ H^2( \mathbb B) $.
In fact, the space decomposition  (\ref{HH--}) shows
\[ f=f^+ +f^-,\]
where $ f^+\in H_+^2( \mathbb T)  $ and $ f^-\in H_-^2( \mathbb T)  $ and then the Cauchy formula shows
\[ \mathcal Sf(q)=f^+(q) \in H^2( \mathbb B) .\]
Besides, the conjugte operator of $ S $ is
\[\mathcal S^*f(q):=\left < f , \overline {e(q, \cdot)} \right >.\]
Since $e(\cdot, q)=\overline {e(q, \cdot)}$, we have
\[\mathcal S^*=\mathcal S.\]
\end{remark}

Denote the class of slice normalized Szeg\"o kernels as:
\[ \mathcal D:=\{e_a\vert a\in \mathbb B\}.\]
\begin{theorem}
$ \mathcal D $ is a dictionary of $ H^2( \mathbb B) $,
i.e.
\[\overline{span}_{\mathbb H}\{e_a\vert a\in \mathbb B\}=H^2( \mathbb B).\]
Here the left-hand-side represents  the closure of the right $\mathbb H$-module linear subspace \\
spanned by finite linear combinations of elements in $ \mathcal D $.
\end{theorem}
\begin{proof}
If $ f \in \overline{span}_{\mathbb H}^{\perp}\{e_a\vert a\in \mathbb B\}$, then the reproducing property of $ e_a $ tells us that
$$  f(a) = 0$$ for any $a \in \mathbb B$.  This means $ f=0 $.
\end{proof}

\begin{remark}\label{Unlif}
Notice that  for each $ I\in \mathbb S $,
$$ \mathcal D_{ I} :=\{e_a\vert a\in \mathbb B_I \}  $$
is also a dictionary of $ H^2( \mathbb B) $ as the definition of slice function shows.
However, its slice Takenaka-Malmquist system given by (\ref{TBe}) with coefficients in $ \mathbb B_I $   is not a basis of $ H^2( \mathbb B) $.
\end{remark}

For every $ a \in \mathbb B $, the Blaschke factor (or the M\"obius transfrom) $ B_a  $ is a slice regular function in $ \mathbb B $  defined as

\[B_a(q):= (1-q\bar a)^{-*}*(a-q)\frac{a}{|a|}. \]

\begin{proposition}[\cite{AFS_book}]\label{BB_a}
Let $ a \in \mathbb B $. The Blaschke factor $ B_a $ has the following properties:
\begin{itemize}
\item it takes the unit ball $ \mathbb B $ to itself;
\item it takes the boundary of the unit ball to itself;
\item it has a unique zero point $ a $.
\end{itemize}
\end{proposition}
A Blaschke product is defined to be the *- product of a finite number of Blaschke factors (also see (\cite{AFS_book}):
\[B_{k}(q):={\prod \limits^{*}}_{j=1}^{k} (1-q\bar a_j)^{-*}*(a_j-q)\frac{a_j}{|a_j|},\]
where $ a_k \in \mathbb B$ for any $ k\in \{1, 2,\cdots\}.$

In the unit ball $ \mathbb B $,  the slice rational orthogonal system, i.e.,  the slice $TM$ system,  consists of weighted Blaschke products, i.e. for any $ k\geqslant 1$,
\begin{equation}\label{TBe}
T_k:=B_{k-1}*e_{a_k}.
\end{equation}

\begin{theorem}\label{TM_Th}
$ \{T_k\}_{k\geqslant 1} $ is an orthonormal system.
\end{theorem}
\begin{proof}
By definition, for any $ k\geqslant 1 $ and $ I \in \mathbb S $,
\begin{equation}\label{TT_kk}
 \left < (T_k)_I, (T_k)_I \right>=\frac{1}{2\pi}\int_{0}^{2\pi}\overline{B_{k-1}*e_k(e^{It})}B_{k-1}*e_k(e^{It})dt.
\end{equation}
Equation (\ref{z_wan}) shows
\[ B_{k-1}*e_k(e^{It})=B_{k-1}(e^{It})e_k(e^{Jt}),\]
where $ J\in \mathbb S $ such that  $ e^{Jt}= B_k^{-1}(e^{It})e^{It}B_k(e^{It}) $. So equation (\ref{TT_kk}) becomes
\begin{equation*}
 \left < (T_k)_I, (T_k)_I \right>=\frac{1}{2\pi}\int_{0}^{2\pi}\overline{e_k( e^{Jt})}\ \overline{ B_{k-1}(e^{It}) }B_{k-1}(e^{It})e_k( e^{Jt})dt.
\end{equation*}
Proposition \ref{BB_a}  implies that
\[ \overline{ B_{k-1}(e^{It}) }B_{k-1}(e^{It})=|B_{k-1}(e^{It})|^2=1.\]
By change of variables, the  Cauchy integral formula shows
\begin{equation*}
\begin{aligned}
 \left < (T_k)_I, (T_k)_I \right>=&\frac{1}{2\pi}\int_{0}^{2\pi}\overline{e_k(e^{Jt})}e_k(e^{Jt})dt
\\
=&\frac{\sqrt{1-|a_k|^2}}{2\pi}\int_{\partial \mathbb B_J}(q-a_k)^{-*}(-Jds)e_k(q)
\\
=&\sqrt{1-|a_k|^2}e_k(a_k)
\\
=&1.
\end{aligned}
\end{equation*}
Since the result is independent of imaginary $ I $,  we obtain
$$ \left < T_k, T_k\right>=1.$$

Now we consider the case of different indices where  $1\leqslant l < k$
\begin{equation}\label{TT_kl}
 \left < (T_k)_I, (T_l)_I \right>=\frac{1}{2\pi}\int_{0}^{2\pi}\overline{B_{l-1}*e_l(e^{It})}B_{k-1}*e_k(e^{It})dt.
 \end{equation}
There we  reformulate $ B_{k-1}*e_k $ as
\begin{equation*}
B_{k-1}*e_k:=B_{l-1}*g,
\end{equation*}
where
\[ g(q)=\big({\prod \limits^{*}}_{j=l}^{k-1} (1-q\bar a_j)^{-*}*(a_j-q)\frac{a_j}{|a_j|}\Big)*e_k(q).\]
As before,
\begin{equation}
B_{l-1}*g(e^{It})=B_{l-1}(e^{It})g(e^{Kt}),
\end{equation}
where $ K \in \mathbb S $ such that  $ e^{Kt}= B_l^{-1}(e^{It})e^{It}B_l(e^{It}) $. Notice that
\[ B_{l-1}*e_l(e^{It})=B_{l-1}(e^{It})e_l(e^{Kt}). \]
Then  equation (\ref{TT_kl}) becomes
\[  \left < (T_k)_I, (T_l)_I \right>=\frac{1}{2\pi}\int_{0}^{2\pi}\overline{e_l( e^{Kt})}\ \overline{ B_{l-1}(e^{It}) }B_{l-1}(e^{It})g( e^{Kt})dt. \]
Similarly, we have
\[ |B_{l-1}(e^{It})|^2=1,\]
Again by the change of variables, we apply the Cauchy integral theorem to get

\begin{equation*}
\begin{aligned}
 \left < (T_k)_I, (T_l)_I \right>
=&\frac{1}{2\pi}\int_{0}^{2\pi}\overline{e_l( e^{Kt})}g( e^{Kt})dt
\\
=&\frac{\sqrt{1-|a|^2}}{2\pi}\int_{\partial \mathbb B_K}(q-a_l)^{-*}(-Kds)g(q)
\\
=&\sqrt{1-|a|^2}g(a_l)
\\
=&0,
\end{aligned}
\end{equation*}
which deduces that
\[\langle T_k, T_l \rangle=0.\]
This completes the   proof.
\end{proof}

\section{Slice Adaptive Fourier Decomposition}
We have known that  the slice Hardy space $ H^2(\mathbb B) $ is a reproducing kernel Hilbert space and   $ T_k$  for any $ k\geqslant1 $ is  a slice rational orthogonal system  in the unit ball $ \mathbb B $.

In this section, we intend to adaptively decompose functions in the slice Hardy space into the subspace spanned by $\{T_k\}_{k\ge 1}$ as following:
For any $ f \in H^2(\mathbb B) $ and $ a_1 \in \mathbb B $, there is an equality
\[ f(q)=e_{a_1}(q)\left<f,e_{a_1}\right>+B_{a_1}*\mathcal S_{a_1}f, \]
where
\[ \mathcal S_{a_1}f:=B_{a_1}^{-*}*\big(f(q)-e_{a_1}(q)\left<f,e_{a_1}\right>\big). \]

Denote
\[ r_2(q):=f(q)-e_{a_1}(q)\left<f,e_{a_1}\right>\]
 as the standard remainder and
\[ f_2(q):=S_{a_1}f \]
as the reduced remainder.
By setting $ f_1=f $, we have
\[ f_1(q)=e_{a_1}(q)\left<f_1,e_{a_1}\right>+B_{a_1}*f_2(q). \]

Notice that $ a_1$ is a removable singularity of $ r_2 $. This is because   $ a_1 $ is a common zero point of function $ r_2 $ and  Blaschke factor  $ B_{a_1} $. In fact, Proposition \ref{BB_a} shows $ B_{a_1} $ has the unique zero $ a_1 $. Hence, $ f_2 $ is a slice regular function in $\mathbb B $.
Meanwhile, by the right $\mathbb H$-linear properties of the inner product $ \left<\cdot,\cdot \right>$, we have
\begin{equation}\label{sumre}
\begin{aligned}
\left<e_{a_1}\left<f_1,e_{a_1}\right>,r_2 \right>=&\left<e_{a_1}\left<f_1,e_{a_1}\right>,f_1-e_{a_1}\left<f_1,e_{a_1}\right>\right>
\\
=&\left<e_{a_1}\left<f_1,e_{a_1}\right>,f_1\right>-\left<e_{a_1}\left<f_1,e_{a_1}\right>,e_{a_1}\left<f_1,e_{a_1}\right>\right>
\\
=&\left<e_{a_1},f_1\right>\left<f_1,e_{a_1}\right>-\overline{\left<f_1,e_{a_1}\right>}\left<e_{a_1},e_{a_1}\right>\left<f_1,e_{a_1}\right>
\\
=&0.
\end{aligned}
\end{equation}
This deduces that
\[\|f_2(q)\|^2=\|r_2(q)\|^2=\|f_1(q)\|^2-|\left<f_1,e_{a_1}\right>|^2< + \infty,\]
where the first equality holds as shown in  the proof of Theorem \ref{TM_Th}.
Thus, we have  $ f_2\in H^2(\mathbb B) $.

Now we can apply the   iterative process:
\begin{equation}\label{Ite}
\begin{aligned}
f_1(q)=&e_{a_1}(q)\left<f_1,e_{a_1}\right>+B_{a_1}*f_2(q)
\\
=&e_{a_1}(q)\left<f_1,e_{a_1}\right>+B_{a_1}*(e_{a_2}(q)\left<f_2,e_{a_2}\right>+B_{a_2}*f_3(q))
\\
=&T_1(q)\left<f_1,e_{a_1}\right>+T_2(q)\left<f_2,e_{a_2}\right>+B_{a_1}*B_{a_2}*f_3(q)
\\
=&T_1(q)\left<f_1,e_{a_1}\right>+T_2(q)\left<f_2,e_{a_2}\right>+B_{a_1}*B_{a_2}*(e_{a_3}(q)\left<f_3,e_{a_3}\right>+B_{a_3}*f_4(q))
\\
=&T_1(q)\left<f_1,e_{a_1}\right>+T_2(q)\left<f_2,e_{a_2}\right>+T_3(q)\left<f_3,e_{a_3}\right>+B_{a_1}*B_{a_2}*B_{a_3}*f_4(q)
\\
=&\cdots
\end{aligned}
\end{equation}

Theorem \ref{TM_Th} shows that in the decomposition (\ref{def:TM-102}) the first $n$ terms are orthogonal to each other, so we just need to show the orthogonality between each of the $n$ summed terms and the remainder term. This can be down by following the same method as in the proof of Theorem \ref{TM_Th}.

The orthogonality implies the following energy equality:
\begin{equation}\label{Itenorm}
\|f_1(q)\|^2=|\left<f_1,e_{a_1}\right>|^2+|\left<f_2,e_{a_2}\right>|^2+\cdots +|\left<f_k,e_{a _k}\right>|^2+ \|f_{k+1}(q)\|^2
\end{equation}

Now we have had  a decomposition of  $ f \in H^2(\mathbb B ) $. We want to know that whether it is convergent and how fast it converges as $k\to\infty$. Clearly, the answer relies on the choice of $ a_n $. Our purpose is to find at every $n$-th step a suitable parameter $ a_n $ such that the corresponding normalized Szeg\"o kernel $e_{a_n}$ extracts out the largest possible energy portion from the reduced remainder $f_n.$ The premise is that the maximal choice must exist.

\begin{theorem}[maximum selection principle]\label{max-102}
For any $ f\in H^2(\mathbb B ) $, there exists an element $ a \in \mathbb B $ such that
\[ |\left<f,e_a\right>|=\max\limits_{b\in\mathbb B}|\left<f,e_b\right>|.\]
\end{theorem}
\begin{proof}
We only need to prove that
\[\lim_{|a|\to1}|\left<f,e_a\right>|=\sqrt{1-|a|^2}|f(a)|=0.\]
In fact,  Theorem \ref{H2} implies that there exists a polynomial $g$ defined in the closure of $ \mathbb B $ such that
\[ \|f-g\|<\frac{\varepsilon}{2}.\]
The inner product is then divided into two parts:
\[\left<f,e_a\right>=\left<f-g,e_a\right>+\left<g,e_a\right>.\]
 The Cauchy-Schwarz inequality (\ref{C-S}) implies
\[\vert\left<f-g,e_a\right>\vert\leqslant \|f-g\| < \frac{\varepsilon}{2}.\]
Now let $ C $ be any but fixed bound of $g$ on the closed unit disc. When $|a|$ is sufficiently close to $1,$ there follows
\[|\left<g,e_a\right>|=\sqrt{1-|a|^2}|g(a)|\leqslant C\sqrt{1-|a|^2}<\frac{\varepsilon}{2}.\]
Combining the above two estimates the proof is complete.

\end{proof}

\begin{lemma}\label{frT}
 With the notation  $ f_n,\ r_n, \ T_n, e_{a_n} $ defined in the text for $ n\geqslant 1 $ there hold
\[\left<f_n,e_{a_n}\right>=\left<r_n,T_n\right>=\left<f,T_n\right>.\]
\end{lemma}
\begin{proof}
For the first equality, recall that
\[ r_n=B_{n-1}*f_n, \quad  T_n=B_{n-1}\sqrt{}*e_n.\]
Following the proof of Theorem \ref{TM_Th}, we obtain
\begin{equation*}
\begin{aligned}
 \left<r_n,T_n\right>=&\left<B_{n-1}*f_n,B_{n-1}*e_n\right>
 \\
 =&\left<f_n,e_{a_n}\right>.
\end{aligned}
\end{equation*}
For the second equation, Theorem \ref{TM_Th} shows
\[ \left<T_k, T_n\right> = 0,  \quad 1\leqslant k<n.\]
Hence, the iteration formula (\ref{Ite}) deduce
\begin{equation*}
\left<f, T_n\right>=\left<r_n,T_n\right>.
\end{equation*}
\end{proof}

\begin{theorem}
Let $ f\in H^2(\mathbb B). $ If for every $ n\geqslant 1 $, the parameters $ a_n $ is chosen according to the maximal selection principle in Theorem \ref{max-102}, then
\[ f=\sum\limits_{n=1}^{\infty}T_n\left<f_n,e_{a_n}\right>=
\sum\limits_{n=1}^{\infty}T_n\left<f,T_n\right>.\]
\end{theorem}

\begin{proof}
In the last two sections, we have introduced the slice Hardy space $ H^2(\mathbb B) $ of  the slice regular functions over the quaternion field. When equipped with the inner product $\left<\cdot, \cdot \right>$, it is a  quaternion Hilbert space. By virtue of the Cauchy formula, we obtained  the slice normalized Szeg\"o kernel $ \{e_a\}_{a\in \mathbb B} $.
Together with $ e_a$, the slice Blaschke products and $*$-product, we established the theory of TM systems $ \{T_k\}_{k\ge1} $ in the slice regular Hardy space. With these preparations and the maximum selection principle we can translate the proof of Theorem 2.2 in \cite{Qian_1} word by word to get the counterpart convergence result in the slice regular Hardy space context.
\end{proof}

\section{The convergence rate}
In this section, we prove a convergence rate result for slice  adaptive Fourier decomposition. We consider the convergence rate issue in a subclass of $H^2(\mathbb B) $, defined by:
\[ H^2(\mathcal D, M):= \{ f\in H^2(\mathbb B):f=\sum \limits_{k=1}^{\infty}  c_ke_{b_k}, \ e_{b_k}\in \mathcal D, \
\sum \limits_{k=1}^{\infty} |c_k| \leqslant M \}, \]
where $ \mathcal D $ is the dictionary consisting of the slice normalized Szeg\"{o} kernels and $ M $ is a positive constant.

\begin{lemma}
If $ f\in H^2(\mathcal D, M) $, then $ \|f\|\leqslant M $.
\end{lemma}
\begin{proof}
Since $ f\in H^2(\mathcal D, M) $, there exist a quaternion series $ \{c_k\}_{k\geqslant1} $ and a function series  $ \{ e_{b_k}\}_{k\geqslant1}  \in \mathcal D $ such that
\[ f=\sum_{k=1}^{\infty}c_ke_{b_k}, \quad \textit{with} \quad \sum_{k=1}^{\infty}|c_k|\leqslant M. \]
Thus we obtain
\begin{equation*}
\begin{aligned}
\|f\|^2=&|\left <f, f\right >|
\\
=&\left|\left<f, \sum_{k=1}^{\infty}c_ke_{b_k}\right>\right|
\\
\leqslant& \sum_{k=1}^{\infty}|c_k||\left<f, e_{b_k}\right>|
\\
\leqslant& \|f\| \sum_{k=1}^{\infty}|c_k|
\\
\leqslant& M\|f\|,
\end{aligned}
\end{equation*}
where the second inequality holds due to the Cauchy-Schwarz inequality (\ref{C-S}).
\end{proof}

\begin{lemma}[\cite{Devore_1}]\label{ser_2}
Let $ A $ be a positive constant and $ \{d_m\}_{m=1}^{\infty} $ be a series of non-negative numbers satisfying
\[d_1\leqslant A, \quad d_{m+1}\leqslant d_m(1-\frac{d_m}{A}),\]
then for every positive integer $ m $, we have
\[ d_m \leqslant \frac{A}{m}.\]
\end{lemma}

Now we can show that the   convergence rate is about $O(m^{-1/2})$ in the space $ H^2(\mathcal D, M) $.
\begin{theorem}
If $ f \in H^2(\mathcal D, M) $, then
\[ \| r_m\| \leqslant \frac{M}{\sqrt{m}}.\]
\end{theorem}
\begin{proof}
 The proof follows the same route as for the classical complex Hardy space. However, we write down the proof since there are details which make it not a direct translation.

Since  $ f\in H^2(\mathcal D, M) $,   there exist $ \{c_k\}_{k\geqslant1}\in \mathbb H $ and
$ \{b_k\}_{k=1}^{\infty} \in \mathbb B $ such that
\begin{equation}\label{fDM}
f=\sum \limits_{k=1}^{\infty} e_{b_k}c_k,\quad \textit{with} \quad \sum_{k=1}^{\infty}|c_k|\leqslant M.
\end{equation}
It follows from   (\ref{Ite}),  (\ref{Itenorm}),  and Lemma \ref{frT} that
\begin{equation}\label{rm}
 \|r_{m+1}\|^2=\|r_m\|^2-|\left<f_m,e_{a_m}\right>|^2=\|r_m\|^2-|\left<r_m,T_m\right>|^2.
\end{equation}

Now we consider the second term in the right side of equality (\ref{rm}). By applying the maximum selection principle of the $m$-th step, Lemma \ref{frT} and the reproducing property of $ e_{a} $,   we have
\begin{equation}\label{rm2}
\begin{aligned}
|\left<r_m, T_m\right>|=&\sup_{a \in \mathbb B}|\left <r_m,T_{\{a_1,\cdots, a_{m-1},a\}}\right>|
\\
=&\sup_{a \in \mathbb B}|\left <f_m, e_{a}\right>|
\\
=&\sup_{a \in \mathbb B}\sqrt{1-|a|^2}|B_{m-1}^{-*}*r_m(a)|
\\
=&\sup_{ a \in \mathbb B}\sqrt{1-|\hat a|^2}|B_{m-1}^{-1}(\hat {a})||r_m(\hat {a})|
\\
\geqslant&\sup_{[b_k]}\sqrt{1-|\hat b_k|^2}|B_{m-1}^{-1}(\hat {b_k})||r_m(\hat {b_k})|
\\
\geqslant&\sup_{[b_k]}\sqrt{1-| b_k|^2}|r_m( {b_k})|
\\
\geqslant&\sup_{b_k}\sqrt{1-|b_k|^2}|r_m(b_k)|,
\end{aligned}
\end{equation}
where the fourth equality holds because of formula (\ref{z_wanw}) with $$ \hat a=B^c_{m-1}(a)^{-1}a B^c_{m-1}(a)\in [a].$$
In the first inequality, we consider the set $ \{[b_k]_{k\geqslant1}\} $  which is the spherical extension of the set $ \{b_k\}_{k\geqslant1} $.

We claim that
\[\sup_{b_k}\sqrt{1-|b_k|^2}|r_m(b_k)|\geqslant\frac{1}{M}\|r_m\|^2.\]
In fact, by applying the orthogonal iterative process of $ f $, we obtain
\[ \left<r_m,r_m\right>=\big| \langle r_m,f\rangle-\langle r_m,\sum_{k=1}^{m-1}T_k\left<f_k,e_{a_k}\right>\rangle\big|=|\left<r_m,f\right>|.\]
The equation (\ref{fDM}) and  the reproducing property of $ e_{b_k} $ give
\begin{equation}\label{rm1}
\begin{aligned}
|\left<r_m,f\right>|
=&\big|\langle r_m, \sum \limits_{k=1}^{\infty} e_{b_k}c_k\rangle \big|
\\
\leqslant&M\sup_{b_k}|\left<r_m, e_{b_k}\right>|
\\
=&M\sup_{b_k}\sqrt{1-|b_k|^2}|r_m(b_k)|,
\end{aligned}
\end{equation}
The claim is hence verified.

By substituting   (\ref{rm1}) and   (\ref{rm2}) into equality (\ref{rm}), we have
\[\|r_{m+1}\|^2\leqslant \|r_m\|^2(1-\frac{\|r_m\|^2}{M^2}).\]
Applying Lemma \ref{ser_2}, we obtain

\[\|r_{m}\|^2 \leqslant \frac{M^2}{m}.\]

\end{proof}

\section{The slice Hardy space over $ \mathbb H^+ $}
The slice Hardy space over $ \mathbb H^+ $ has also been studied in \cite{AFS_book}. Again, we use an equivalent definition.
Let $ H^+ $ be the right half plane of $ \mathbb H $, i.e.
\[ H^+:=\{ q \in \mathbb H \  \bm{\big\vert} \ Re(q)>0 \}.\]
In this section, we just list the corresponding results of the slice Hardy space  over $ H^+ $,
for which the proofs are similar to those for the slice Hardy space over $ \mathbb B $.

\begin{definition}
The slice Hardy space  $ H^2(\mathbb H^+) $  consists of slice regular functions $ f $, which satisfies
\[ \|f\|^2:=\frac{1}{2\pi}\int_{T^2}\sin \theta_1d\theta\int_{-\infty}^{+\infty}|f(I(\theta)y)|^2dy<\infty,   \]
where $ \theta=(\theta_1, \theta_2)\in T^2= [0,\pi]^2$. Furthermore,  $ H^2(\mathbb H^+) $  is a Hilbert space.
\end{definition}

%\begin{remark}
Denote by  $ \left < \cdot, \cdot \right> $ the  inner product with the induced   norm $ \|\cdot\| $.
%\end{remark}

For any $ a \in \mathbb H^+ $,  the slice normalized Szeg\"o kernel is  a left slice regular function over $ \mathbb H^+ $,
defined as
\[ e_a(q):=\sqrt{\frac{Re(a)}{\pi}}(q+\bar a)^{-*},  \]
where $ -* $ is the regular reciprocal in Definition \ref{*} and $ Re(a) $ is the real part of $ a\in \mathbb H $.
Then $ e_a $ is a reproducing kernel of $ H^2(\mathbb H^+) $, i.e.
\begin{equation}
\left < f, e_a \right>=\sqrt{4\pi Re(a)}f(a).
\end{equation}
Besides, the slice normalized Szeg\"o kernel $ e_a $ is a dictionary of $ H^2(\partial \mathbb H^+) $,  i.e.
\[\overline{span}_{\mathbb H}\{e_a\vert a\in \mathbb H\}=H^2( \partial \mathbb H^+),\]

The Blaschke product is
\[B_k(q):={\prod \limits^{*}}_{j=1}^{k} (q+\bar a_j)^{-*}*(q-a_j),\]
where $ a_k \in \mathbb H_+$ for any $ k\in \{1, 2,\cdots\}. $

In the right half plane $ \mathbb H_+ $,  the slice  rational orthogonal system (i.e. the  slice Takenaka-Malmquist  system) consists of weighted Blaschke product, i.e. for any $ k\geqslant 1$,
\[T_k:=B_{k-1}*e_k.\]

\begin{theorem}\label{TM_Th_1}
$ \{T_k\}_{k\geqslant 1} $ is a normal orthogonal system, i.e.
\begin{equation*}
\begin{cases}  \left < T_k, T_k \right>=1, \quad  k\geqslant 1, \\
\left < T_k, T_l \right>=0, \quad 1\leqslant l < k.
\end{cases}
\end{equation*}
\end{theorem}

Now we consider the slice adaptive Fourier decomposition.
For any $ f \in H^2(\mathbb H^+) $ and $ a_1,a_2, a_3 \cdots \in \mathbb H^+ $, there is an iterative process:
\begin{equation}
\begin{aligned}
f_1(q)=&e_{a_1}(q)\left<f_1,e_{a_1}\right>+B_{a_1}*f_2(q)
\\
=&e_{a_1}(q)\left<f_1,e_{a_1}\right>+B_{a_1}*(e_{a_2}(q)\left<f_2,e_{a_2}\right>+B_{a_2}*f_3(q))
\\
=&T_1(q)\left<f_1,e_{a_1}\right>+T_2(q)\left<f_2,e_{a_2}\right>+B_{a_1}*B_{a_2}*f_3(q)
\\
=&T_1(q)\left<f_1,e_{a_1}\right>+T_2(q)\left<f_2,e_{a_2}\right>+B_{a_1}*B_{a_2}*(e_{a_3}(q)\left<f_3,e_{a_3}\right>+B_{a_3}*f_4(q))
\\
=&T_1(q)\left<f_1,e_{a_1}\right>+T_2(q)\left<f_2,e_{a_2}\right>+T_3(q)\left<f_3,e_{a_3}\right>+B_{a_1}*B_{a_2}*B_{a_3}*f_4(q)
\\
=&\cdots
\end{aligned}
\end{equation}
Orthogonality  of the Takenaka-Malmquist system $ \{T_k\}_{k\geqslant 1} $ implies the following energy relation:
\begin{equation*}
|f_1(q)|^2=|\left<f_1,e_{a_1}\right>|^2+|\left<f_2,e_{a_2}\right>|^2+\cdots +|\left<f_k,e_{a_k}\right>|^2+ |f_{k+1}(q)|^2.
\end{equation*}

\begin{theorem}\label{max}
For any $ f\in H^2(\mathbb H^+ ) $, there exists an element $ a \in \mathbb H^+ $ such that
\[ |\left<f,e_a\right>|=\max\limits_{b\in\mathbb H^+}\{|\left<f,e_b\right>|\}.\]
\end{theorem}

\begin{theorem}
Let $ f\in H^2(\mathbb H^+) .$ If for every $ n\geqslant 1 $ the parameters $ a_n $  in relation to the reduced remainder function $ f_n $ is chosen according to Theorem \ref{max}, then
\[ f=\sum\limits_{n=1}^{\infty}T_n\left<f,e_{a_n}\right>=\sum\limits_{n=1}^{\infty}T_n\left<f,T_n\right>.\]
\end{theorem}

Denote
\[ H^2(\mathcal D, M):= \{ f\in H^2(\mathbb H^+):f=\sum \limits_{k=1}^{\infty}  e_{b_k}c_k, \ e_{b_k}\in \mathcal D, \
\sum \limits_{k=1}^{\infty} |c_k| \leqslant M \}, \]
where $ \mathcal D $ is the dictionary consisting of the slice normalized Szeg\"{o} kernels and $ M $ is a positive constant.
Then we have the following
\begin{theorem}
If $ f \in H^2(\mathcal D, M) $, then
\[ \| r_m\| \leqslant \frac{M}{\sqrt{m}}.\]
\end{theorem}

\bibliographystyle{amsplain}

\end{document}